\documentclass[12pt]{amsart}
\usepackage{verbatim}
\usepackage{amssymb, amsmath}
\usepackage{graphicx}
\usepackage{caption}
\usepackage{subcaption}

\usepackage{color}

\DeclareMathOperator{\arcsinh}{arcsinh}

\begin{document}
\newtheorem{thm}{Theorem}[section]
\newtheorem*{thm*}{Theorem}
\newtheorem{lem}[thm]{Lemma}
\newtheorem{prop}[thm]{Proposition}
\newtheorem{cor}[thm]{Corollary}
\newtheorem*{conj}{Conjecture}
\newtheorem{proj}[thm]{Project}
\newtheorem{question}[thm]{Question}
\newtheorem{rem}{Remark}[section]

\theoremstyle{definition}
\newtheorem*{defn}{Definition}
\newtheorem*{remark}{Remark}
\newtheorem{exercise}{Exercise}
\newtheorem*{exercise*}{Exercise}

\numberwithin{equation}{section}

\newcommand{\rad}{\operatorname{rad}}

\newcommand{\Z}{{\mathbb Z}} 
\newcommand{\Q}{{\mathbb Q}}
\newcommand{\R}{{\mathbb R}}
\newcommand{\C}{{\mathbb C}}
\newcommand{\N}{{\mathbb N}}
\newcommand{\FF}{{\mathbb F}}
\newcommand{\fq}{\mathbb{F}_q}
\newcommand{\rmk}[1]{\footnote{{\bf Comment:} #1}}

\renewcommand{\mod}{\;\operatorname{mod}}
\newcommand{\ord}{\operatorname{ord}}
\newcommand{\TT}{\mathbb{T}}
\renewcommand{\i}{{\mathrm{i}}}
\renewcommand{\d}{{\mathrm{d}}}
\renewcommand{\^}{\widehat}
\newcommand{\HH}{\mathbb H}
\newcommand{\Vol}{\operatorname{vol}}
\newcommand{\area}{\operatorname{area}}
\newcommand{\tr}{\operatorname{tr}}
\newcommand{\norm}{\mathcal N} 
\newcommand{\intinf}{\int_{-\infty}^\infty}
\newcommand{\ave}[1]{\left\langle#1\right\rangle} 
\newcommand{\Var}{\operatorname{Var}}
\newcommand{\Prob}{\operatorname{Prob}}
\newcommand{\sym}{\operatorname{Sym}}
\newcommand{\disc}{\operatorname{disc}}
\newcommand{\CA}{{\mathcal C}_A}
\newcommand{\cond}{\operatorname{cond}} 
\newcommand{\lcm}{\operatorname{lcm}}
\newcommand{\Kl}{\operatorname{Kl}} 
\newcommand{\leg}[2]{\left( \frac{#1}{#2} \right)}  
\newcommand{\Li}{\operatorname{Li}}

\newcommand{\sumstar}{\sideset \and^{*} \to \sum}

\newcommand{\LL}{\mathcal L} 
\newcommand{\sumf}{\sum^\flat}
\newcommand{\Hgev}{\mathcal H_{2g+2,q}}
\newcommand{\USp}{\operatorname{USp}}
\newcommand{\conv}{*}
\newcommand{\dist} {\operatorname{dist}}
\newcommand{\CF}{c_0} 
\newcommand{\kerp}{\mathcal K}

\newcommand{\Cov}{\operatorname{cov}}
\newcommand{\Sym}{\operatorname{Sym}}

\newcommand{\Ht}{\operatorname{Ht}}

\newcommand{\E}{\operatorname{\mathbb E}} 
\newcommand{\sign}{\operatorname{sign}} 
\newcommand{\meas}{\operatorname{meas}} 
\newcommand{\length}{\operatorname{length}} 

\newcommand{\divid}{d} 

\newcommand{\GL}{\operatorname{GL}}
\newcommand{\SL}{\operatorname{SL}}
\newcommand{\re}{\operatorname{Re}}
\newcommand{\im}{\operatorname{Im}}
\newcommand{\res}{\operatorname{Res}}
 \newcommand{\eigen}{\Lambda} 
\newcommand{\tens}{\mathbf t} 
\newcommand{\diam}{\operatorname{diam}}
\newcommand{\fixme}[1]{\footnote{Fixme: #1}}
 \newcommand{\EWp}{\mathbb E^{\rm WP}_g} 
\newcommand{\orb}{\operatorname{Orb}}
\newcommand{\supp}{\operatorname{Supp}}
\newcommand{\mmfactor }{\textcolor{red}{c_{\rm Mir}}}
\newcommand{\Mg}{\mathcal M_g} 
\newcommand{\MCG}{\operatorname{Mod}} 
\newcommand{\Diff}{\operatorname{Diff}} 
\newcommand{\If}{I_f(L,\tau)}

\title[GOE statistics on moduli space of surfaces]
{GOE statistics on the moduli space of surfaces of large genus}
\author{Ze\'ev Rudnick  }
\address{School of Mathematical Sciences, Tel Aviv University, Tel Aviv 69978, Israel} 
\email{rudnick@tauex.tau.ac.il}
\date{\today}
\begin{abstract}
 For a compact hyperbolic surface, we define a smooth linear statistic, mimicking the number of Laplace eigenvalues in a short energy window.  
 We study the variance of this statistic, when averaged over the moduli space $\Mg$ of all genus $g$ surfaces  with respect to 
 the Weil-Petersson measure. We show that in the double limit, first taking the large genus limit and then the short window limit, we recover GOE statistics for the variance.  The proof makes essential use of Mirzakhani's integration formula.
\end{abstract}
\maketitle

\section{Introduction}

 \subsection{Motivation}
An outstanding conjecture in quantum chaos is that the statistics of the energy levels of ``generic'' chaotic systems with time reversal symmetry  are described by those of  the Gaussian Orthogonal Ensemble (GOE) in Random Matrix Theory \cite{BGS}. 
This conjecture  seems to be extremely difficult, with no single case being proved. It has long been desired to improve the situation by averaging over  a suitable ensemble of chaotic systems, see e.g. the discussion in \cite{NZ}, and \cite{AS} for a numerical study, averaging over $30$ hyperbolic surfaces of genus $2$.  
So far this has not been successfully implemented, in part because of the lack of mechanisms to execute averaging. 
In this paper we carry out a version of such  ensemble averaging 
on the moduli space $\Mg$ of compact hyperbolic surfaces,  equipped with the Weil-Petersson measure, using the pioneering work of  Mirzakhani.  

With this ensemble averaging, we examine the rigidity of the eigenvalue spectrum $\{\lambda_j\}_{j=0}^\infty$ of hyperbolic surfaces. The term ``rigidity'' refers to slow growth of the variance of the    number  of eigenvalues in an energy window $ [E, E+W]$,  
with $E,W\to \infty$, $W=o(\sqrt{E})$, or as in this paper,  of smooth linear statistics   mimicking the count of eigenvalues in   windows.
  The reason that we choose the eigenvalue window of this form is that in this regime,  
Berry \cite{Berry1985, Berry1986} argued that the fluctuations   are universally those of the GOE, 
    but cease to be universal for larger windows. 
    In detail, if $n(E;W)$ is the number of eigenvalues of a fixed hyperbolic surface $X$ in the window $[E,E+W]$, then by Weyl's law, on average we have 
    \[
    \bar n:=\langle n(E;W) \rangle \sim \frac{\area(X)}{4\pi} \cdot W
    \] where $\langle \bullet \rangle$ denotes an average over a range of energies $E$ (the exact specifics of the averaging are immaterial). Berry then considers the number variance 
\[
\Sigma^2(\bar n) = \langle \left| n(E;W)-\bar n\right|^2 \rangle
\]    
and his conjecture, specialized to our context,   is that it should behave like the corresponding quantity in the GOE, namely 
\[
\Sigma^2(\bar n)\sim \frac {2}{\pi^2} \log \bar n .
\]
  No instance of this has been proved to date, though arithmetic surfaces were found to be exceptions to this rule, 
  see \cite{BGGS, LS, RudnickCLT} and the survey \cite{Marklof AQC}. Our main goal is to show that, after averaging over the moduli space $\Mg$, GOE statistics hold in a suitable limit for a smooth version of the number variance. 

\subsection{Our results}
Let $X$ be a compact hyperbolic surface of genus $g\geq 2$, 
and $\lambda_j = 1/4+r_j^2$ be the eigenvalues of the Laplacian on $X$, where the spectral parameter $r_j$, defined up to a sign, lies in $\R\cup [-i/2,i/2]$, to make $\lambda_j\geq 0$.  
 For an even test function $f$ with compactly supported Fourier transform $\^f \in C_c^\infty(\R)$ and $\tau>0$, $L>1$,  
 define the smooth linear statistic
\[
N_{f,L,\tau}(X):= \sum_{j\geq 0} f\left(L\left(r_j-\tau \right)\right) +  f\left(L\left(r_j + \tau \right)\right)  . 
\]
This is a smooth count of the number levels in a frequency window of width\footnote{In some of the older literature, the letter $L$ is reserved for the expected number of levels in the window.} $1/L$ about the fixed frequency $\tau$, equivalently eigenvalues in a 
window of width $W=2\tau/L$ around the energy $E=\tau^2+1/4+1/(2L)^2$.  
Further, set 
\[
\bar N:=\bar N_{f,L,\tau}=2(g-1)\intinf f(L(r-\tau)) r\tanh(\pi r)dr. 
\]

Weyl's law in this context is that for a fixed surface, as $\tau,L \to \infty$, and $L=o(\log \tau)$, we have if $\intinf f(x)dx \neq 0$ 
(see \S~\ref{sec:reduction})
\[
N_{f,L,\tau}(X) \sim  \bar N\sim   (g-1)\intinf f(x)dx \frac{2\tau}{L}.
\] 

For the corresponding smooth linear statistics in the GOE, the variance was computed by Dyson and Mehta \cite[Section II]{Dyson-Mehta} to be 
\[
\Sigma^2_{\rm GOE}(f):=2\intinf |x| \^f(x)^2 dx 
\]
(for the Gaussian Unitary Ensemble, the factor $2$ is dropped). 
Throughout the paper, we use the normalization 
\[
\^f(x) =\frac 1{2\pi} \intinf f(y)e^{-ix y}dy,
\]
 so that $f(y) = \intinf \^f(x)e^{ixy}dx$. 

We study the expectation and variance of $N_{f,L,\tau}$, when averaged over the moduli space $\Mg$ of all genus $g$ surfaces  with respect to the Weil-Petersson measure (see \S~\ref{sec:Mg} for relevant terminology and background).  

For the expectation, we show
\[
\lim_{g\to \infty} \left( \EWp\left(N_{f,L,\tau}\right) - \bar N   \right) =\If 
\]
where 
\[
\If:=\frac{4}{L}\int_0^\infty  \sum_{k=1}^\infty \^f\left(\frac {kx}L \right)\frac{\sinh^2(x/2)}{\sinh(kx/2)} \cos(k\tau x) dx.
\]
So when $g\to \infty$ with $\tau>0$, $L>1$ fixed, the expected number of levels counted by $N_{f,L,\tau}$ is of order $g$, with $\If$ a lower order term, independent of $g$. 
Note that\footnote{The notation $f\ll g$ means $f=O(g)$, and $f\ll_A g$ means that the implied constant depends on the parameter $A$.} 
\[
\If \ll   \frac{e^{L/2}}{L \tau }+\frac 1L \left( 1+\frac 1\tau \right)  
\]
   so that when $L,\tau\to \infty$, $L=o(\log \tau)$ then $\If \to 0$.

Our main object of  study is the variance 
\[
\Sigma^2_g(\tau,L;f):=  \EWp\left( \left|N_{f,L,\tau}-\EWp(N_{f,L,\tau}) \right|^2 \right) .  
\]

We show that  the large genus limit $g\to \infty$ of $\Sigma^2_g(\tau,L;f)$ exists, 
and after taking the short window limit   $L \to \infty$ we recover the GOE result:  
 
\begin{thm}\label{main thm}
Fix $\tau>0$. 
Then in the large genus limit $g\to \infty$, we have
\[
\left( \lim_{g\to \infty}   \Sigma^2_g(\tau,L;f) \right) = \Sigma^2_{\rm GOE}(f)  + O\left( \frac{\log L}{L^2} \right) .
\]
Hence in the double limit, we recover GOE statistics for the variance $\Sigma^2_g(\tau,L;f)$:
\[
\lim_{  L \to \infty }\left( \lim_{g\to \infty}   
\Sigma^2_g(\tau,L;f) \right) = \Sigma^2_{\rm GOE}(f) .
\]
\end{thm}
  
 In view of the above, it is natural to expect that for {\em fixed} $g> 2$, 
for {\em almost all} $X\in \Mg$ (w.r.t. the Weil-Petersson measure), 
the energy variance of the linear  statistic $N_{f,L,\tau}$ coincides with that of GOE. 


 \subsection{About the proof}  

  We use Selberg's trace formula to express the linear statistic  $N_{f,L,\tau}$ as a smooth main term $\bar N$ and a sum $N^{osc}$ over closed geodesics of $X$. The particular choice of the linear statistic restricts the sum to closed geodesics of length at most $L$. We further break up the sum to a sum $N_{\rm sns}$ over simple (i.e. having no self-intersections), nonseparating geodesics (i.e. those simple geodesics $\gamma$ such that $X\backslash \gamma$ is connected), a sum $N_{\rm SSep}$ over simple separating geodesics, and a sum $N'$ over non-simple geodesics. 
  
The expected value of the sum over simple non-separating geodesics  is given by 
\[
\lim_{g\to \infty} \EWp(N_{\rm sns})  = \If   .
\]
    The expected values of the other two sums $N_{\rm SSep}$ and $N'$ vanish in the limit $g\to \infty$.  
    We find
  \[
  \lim_{g\to \infty} \EWp\left( N_{f,L,\tau} -\bar N -\If\right) = 0 .
  \] 
  
  
  The variance $\Sigma^2_g(\tau,L;f) $ involves pairs of closed geodesics, and Mirzakhani's integration formula is used to evaluate averages over $\Mg$ of the sum over {\em simple} pairs of  geodesics, that is pairs of disjoint simple closed geodesics, and among those it is the sum $N_{{\rm sns},2}$ over non-separating pairs (i.e. those simple pairs $(\gamma, \gamma')$ for which $X\backslash \gamma\cup \gamma'$ is connected) which give the dominant contribution in the large genus limit $g\to \infty$: 
  \[
\lim_{g\to \infty}  \EWp(N_{{\rm sns},2})=\Sigma^2_{\rm GOE} (f)+ \If^2   + O\left( \frac{\log L}{L^2} \right).
  \] 
 Here $ \Sigma^2_{\rm GOE} (f)$ comes from diagonal pairs, while the term $\If^2 $ comes from  the off-diagonal pairs.

The contribution of simple separating pairs of  geodesics is bounded by 
\[
 \EWp(N_{{\rm SSep},2}) \ll_L  \frac 1g 
\]
 which vanishes in the limit $g\to \infty$.

 The contribution $N''$ of non-simple pairs of geodesics is not covered by the integration formula, and we use a mixture of considerations to bound the expected value $\EWp(N'')$. For the sum over pairs of geodesics which are both non-simple, or are intersecting, we use a collar lemma to show that the sum is uniformly bounded in terms of the number $Y'_{2,g}$ of such terms and then rely on a bound for $\EWp(Y'_{2,g})$ provided by Mirzakhani and Petri \cite{MP}. For pairs of disjoint geodesics where one is simple and the other is not, we do not have such a uniform bound and we use a variant of the above argument. In total we find $\EWp(N'')\ll_L 1/\sqrt{g}$.   
  
  Putting all these together gives
  \[
  \begin{split}
\lim_{g\to \infty}  \Sigma^2_g(\tau,L;f) &=\lim_{g\to \infty} \EWp( (N^{osc})^2)) - \If^2
 \\
 &=  \Sigma^2_{\rm GOE} (f)   + O\left( \frac{\log L}{L^2} \right)
 \end{split} \]
  and taking $L\to \infty$ gives Theorem~\ref{main thm}. 
  
 
 \subsection{Related work on spectral theory on $\Mg$}

There has been much interest recently in the spectral theory of random surfaces of large genus. One direction was to  give 
 a lower bound for the first eigenvalue $\lambda_1(X)$ for a typical surface $X\in \Mg$ of large genus, showing $\lambda_1(X)>3/16-\varepsilon$  with probability tending to one as $g\to \infty$; here probability is with respect to the Weil-Petersson measure  \cite{LW, WX, Hide}. A similar result was proved for a different model of random curves, namely random covers of large degree, in \cite{MNP}. Monk \cite{Monk} gives bounds on the number of ``exceptional'' eigenvalues $\lambda_j(x)<1/4$ for ``typical'' surfaces of large genus. In a different direction, \cite{GLPT} give bounds for the $L^p$ norms of eigenfunctions for typical surfaces of large genus.

   \subsection{Acknowledgments}
   We thank Jon Keating,  Bram Petri, Shvo Regavim, Igor Wigman, Ouyang Zexuan, and the referee for several comments and corrections, and to Omer Rudnick for the figures.  
   
   This research was supported by the European Research Council (ERC) under the European Union's Horizon 2020 research and innovation programme (grant agreement No. 786758) and by the Israel Science Foundation (grant No. 1881/20).

 \section{Background on Mirzakhani's integration formula}\label{sec:Mg}
 The goal of this section is to present Mirzakhani's integration formula, which allows to integrate certain ``geometric functions'' over the moduli space $\Mg$. For further background, see \cite{Buser, FM, Wright BAMS}. 

\subsection{Moduli spaces and their volumes}

Let $S_g$ be a smooth compact, connected, oriented surface of genus $g\geq 2$. We denote by $\Diff^+(S)$ the group of orientation preserving diffeomorphisms of $S$, and by $\Diff_0(S)$ the subgroup of those isotopic to the identity. 
Teichm\"{u}ller space $\mathcal T(S_g)$ is the set of hyperbolic structures on $S_g$ 
\[
\mathcal T(S_g)=\{(X,f): f:S_g\simeq X\}/\sim
\]
 where $f:S_g\to  X$ is a diffeomorphism of $S_g$  onto a hyperbolic surface $X$ (a ``marking''), and the equivalence is up to homotopy: $(X,f)\sim (X',f')$ if there is an isometry $h:X\to X'$ so that $(f')^{-1}\circ h \circ f:S_g\to S_g$ is isotopic to the identity.  
 That is, $\mathcal T(S_g)$ is the space of homotopy classes of hyperbolic structures on $S_g$. 
  It is an affine space, of dimension $6g-6$. 

The mapping class group $\MCG(S):=\Diff^+(S)/\Diff_0(S)$ is the group of orientation preserving diffeomorphisms up to isotopy \cite{FM}. 
It is a countable group, and acts properly discontinuously on Teichm\"{u}ller space. 
The moduli space $\Mg$ is the quotient $$\Mg=\mathcal T(S_g)/\MCG(S_g).$$

More generally,  for $n\geq 0$, with $2-2g-n<0$, let $S_{g,n}$ be a smooth compact, connected, orientable surface of genus $g$ with $n$ boundary components.  
Denote by $\MCG(S_{g,n})$  the group of orientation preserving diffeomorphisms which setwise fix the boundary 
components\footnote{The literature has versions of $\MCG(S_{g,n})$ where the requirement is that the boundary is fixed pointwise; we follow the conventions in \cite{MP}.}, 
up to isotopy. 
Given $n$ positive numbers $\vec \ell = (\ell_1,\dots, \ell_n)$, let $\mathcal T(S_{g,n},\vec \ell)$ be the space of hyperbolic structures on $S_{g,n}$ with geodesic boundary components of lengths $\ell_1,\dots, \ell_n$.  
Then $\MCG(S_{g,n})$ acts on $\mathcal T(S_{g,n},\vec \ell)$ and the quotient space
\[
\mathcal M_{g,n}(\vec \ell)  = \mathcal T(S_{g,n},\vec \ell)/\MCG(S_{g,n})
\]
is the moduli space of Riemann surfaces of genus $g$ and $n$ geodesic boundary components with lengths given by $\vec \ell$;   when $\vec \ell = (0,\dots,0)$, then $\mathcal M_{g,n}(0,\dots,0)$ is the moduli space of genus $g$ surfaces with $n$ cusps. 

The space $\mathcal T(S_{g,n},\vec \ell)$ has a symplectic form, the Weil-Petersson form, which is invariant under the mapping class group, which induces a volume form ${\rm dVol}^{WP}$   on the moduli space $ \mathcal M_{g,n}(\vec \ell) $. 
 We denote 
\[
V_{g,n}(\ell_1,\dots, \ell_n) = \Vol^{WP} \left(\mathcal M_{g,n}\left(\ell_1,\dots, \ell_n\right)  \right)
\]
and also set 
\[
V_g = \Vol^{WP}(\Mg),  \qquad V_{g,n}=V_{g,n}(0,\dots, 0)  . 
\] 
 
We will need to know volume ratios  \cite[Proposition 3.1]{MP}\footnote{See \cite[footnote page 3]{AM} for a small correction and for more refined asymptotics}:
\begin{equation}\label{ratio V_{g,2}(x,x)/V_g}
\frac{V_{g,n}(\ell_1,\dots, \ell_n)}{V_{g,n}} = \prod_{j=1}^n \frac{\sinh(\ell_j/2)}{\ell_j/2} \cdot 
\left (1+O\left( \frac{(\sum_{j=1}^n \ell_j)\prod_{j=1}^n \ell_j}{g} \right) \right) .
\end{equation}
Furthermore \cite[Theorem 1.4]{MZ}
\begin{equation}\label{ratio V_{g-1,2}/V_g}
\frac{V_{g-1,n+2}}{V_{g,n}} = 1+O\left(\frac {1+n}g \right) .
\end{equation}

 We will also need an estimate on products of volumes   \cite[Lemma 3.2]{MP}: For $q\geq 2$, 
\begin{equation}\label{Gen bound on sum of vols}
\frac 1{V_g}\sum_{ } \prod_{i=1}^q V_{g_i,b_i}  \ll \frac 1{g^{q-1}} 
\end{equation}
where the sum is over all topological types of decompositions of the surface $S_g$ arising from cutting it along $K$ simple disjoint nonisotopic geodesics in $q$ pieces $S_{g_i,b_i}$ having genus $g_i$ and $b_i\geq 1$ boundary components, so that $\sum_{i=1}^q b_i = 2K$ and $\sum_{i=1}^q 2-2g_i -b_i=2-2g$ by the additivity of the Euler characteristic.


\subsection{Mirzakhani's integration formula}

  A closed closed curve on a surface is {\em essential} if it is not contractible, or freely homotopic to one of the boundary components if there are any.  Any essential closed curve on a hyperbolic surface is freely homotopic to a unique geodesic. 
  A closed curve  is {\em simple} if it has no self intersections. 
Fix a simple closed curve $\gamma$ on the base surface $S_g$, and denote by  $\MCG[\gamma]=\{\phi \gamma: \phi\in \MCG\}$ the orbit 
 of $\gamma$ under the mapping class group, that is all curves of the same  ``topological type''  as $\gamma$. The different types are  (see \cite[\S 1.3.1]{FM} and Figure~\ref{fig:cutting surface}): 
\begin{itemize}
\item 
non-separating curves $\gamma_0$, which cut the surface into a surface of signature $(g-1,2)$, i.e. $S_g\backslash \gamma_0$ is a surface of genus $g-1$ with $2$ boundary components, each having length $\ell$.  
\item
For each $i=1,\dots ,\lfloor \frac g2 \rfloor$ the separating curve $\gamma_i$ cutting $S_g$ into two components of  signatures $(i,1)$ and $(g-i,1)$, each having one boundary component of (equal) length $\ell$. 
\end{itemize}

\begin{figure}[ht]
\begin{center}
\includegraphics[height=50mm]{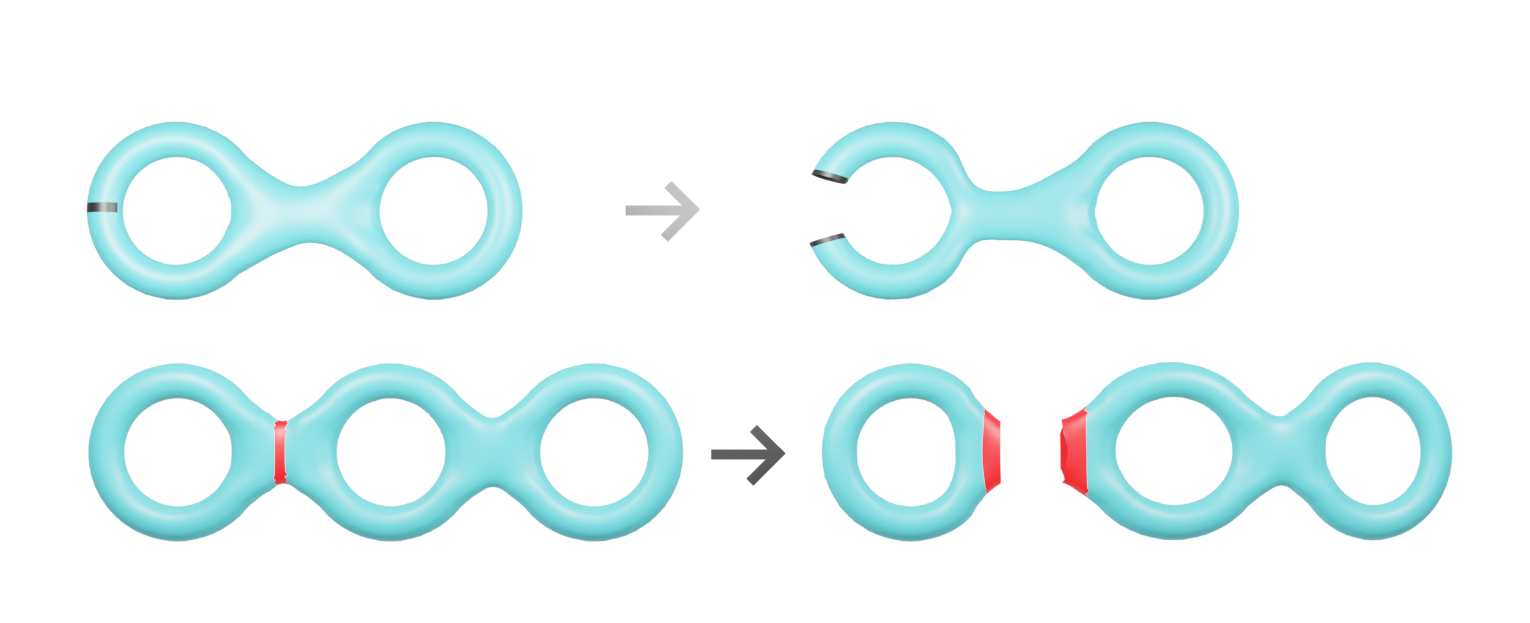}
\caption{ A genus 2 surface $S_{2,0}$ cut by different topological types of geodesics: Top, a non-separating geodesic   gives  a surface $S_{1,2}$ of genus one with two boundary components. 
Bottom, a separating geodesic  cuts the surface into two surfaces $S_{1,1}$ of genus one with one boundary component, and $S_{2,1}$ of genus two with one boundary component.}
\label{fig:cutting surface}
\end{center}
\end{figure}

Given an essential curve $\gamma$ on $S_g$, and a hyperbolic structure $X\in \mathcal T(S_g)$, denote by $\ell_\gamma(X)$  the length (with respect to the metric determined by $X$) of the unique geodesic in the free homotopy class of the  curve $\gamma$.   
Let $f:\R_+\to \R$ be a function on the positive reals. Define $f_\gamma(X)$ to be the sum of $ f\left(\ell_\gamma\left(X\right)\right)$  over the orbit\footnote{So over the cosets $\MCG/{\rm Stab}_{\MCG}(\gamma)$} 
 of $\gamma$ under the mapping class group: 
\[
f_\gamma(X):=\sum_{\alpha \in \MCG[\gamma]} f\left(\ell_\alpha\left(X\right)\right) .
\]
This function is called a geometric function, and is invariant under changing $\gamma$ by $\MCG$, hence descends to the moduli space $\Mg$.


We will need to compute the expected value $$\EWp(f_\gamma) = \frac 1{V_g}\int_{\Mg} f_\gamma(X) {\rm dVol}^{WP}(X) .$$ 
The key to doing so is Mirzakhani's integration formula \cite[Theorem 2.2]{MP}, which says that   the integral of $f_\gamma$ over $\Mg$ is given by   
\begin{equation}\label{MIF one var}
\int_{\Mg} f_\gamma(X) {\rm dVol}^{WP}(X)  = c(\gamma)\int_0^\infty f(\ell) V_{g}(\gamma; \ell) \ell d\ell
\end{equation}
where $0<c(\gamma)\leq 1$ and $V_{g}(\gamma; \ell) $ are determined by the topological type (orbit under $\MCG(S_g)$) of $\gamma$. In particular, if $S_g\backslash \gamma$ is connected (that is $\gamma$ is non-separating), then\footnote{For $g=2$ we get $1$ instead of $1/2$.}
\[
c(\gamma)=\frac 12, \quad g>2
\]  
and
\[
V_{g}(\gamma; x)=V_{g-1,2}(x,x) = \Vol^{WP}(\mathcal M_{g-1,2}(x,x))
\]
is the volume of the moduli space of surfaces of genus $g-1$ with two boundary components, each of length $x$.  If $\gamma$ separates $S_g$ into two pieces: $S_g\backslash \gamma = S_{i,2}\cup S_{g-i,2}$ then each has one boundary component, both of the same length, and the sum of their genera is $g$. In that case
\[
V_{g}(\gamma; x) = V_{i,1}(x)V_{g-i,1}(x) .
\]

More generally, given a multi-curve, that is a $k$-tuple $\Gamma = (\gamma_1,\dots,\gamma_k)$ of disjoint essential simple closed curves, not freely homotopic between themselves, and a function $f:\R_+^k\to \R$, we define 
\[
f_\Gamma(X):= \sum_{\alpha =(\alpha_1,\dots, \alpha_k)  \in \MCG[\Gamma]} 
f\left(\ell_{\alpha_1}\left(X\right),\dots, \ell_{\alpha_k}\left(X\right) \right) 
\]
where $\MCG[\Gamma] = (\phi\gamma_1,\dots, \phi \gamma_k):\phi \in \MCG\}$ is the orbit under the mapping class group. 
For instance, taking pairs of non-homotopic curves $(\gamma,\gamma')$, the orbits of $\MCG(S_g)$ are completely described by the topology of the complement $S_g\backslash \gamma\cup \gamma'$ (see e.g.  \cite[\S 1.3.1]{FM}), in particular there is a unique orbit of 
non-separating pairs, where $S_g\backslash \gamma\cup \gamma' = S_{g-2,4}$ is a surface of genus $g-2$ with $4$ boundary components (Figure~\ref{fig:two bdry}).


 Mirzakhani's integration formula \cite[Lemma 7.3]{MirInv} says that
\begin{equation}\label{gen M integ formula}
\int_{\Mg}f_\Gamma(X) {\rm dVol}^{WP}(X)  =c(\Gamma) \int_{\R_+^k} f(x_1,\dots, x_k) V_g(\Gamma, \mathbf x) \prod_{j=1}^k x_j dx_j
\end{equation}
where $V_g(\Gamma, \mathbf x) $ are volumes of moduli spaces determined by $\Gamma$: if $S_g\backslash \Gamma = \cup_i S_{g_i,n_i}$ then 
\[
V_g(\Gamma, \mathbf x) = \prod_i V_{g_i,n_i}(\mathbf x)
\]  
 and   $0<c(\Gamma)\leq 1$ are constants determined by $\Gamma$, see \cite[footnote on page 368]{Wright BAMS} for an expression. 
 In particular, if $S_g\backslash \cup_{j=1}^k \gamma_j$ is connected then $c(\Gamma) = 1/2^k$.

\section{Reduction to sums over closed geodesics}\label{sec:reduction} 

We recall the Selberg trace formula: 
For a compact hyperbolic surface $X$ of genus $g$, for each $j\geq 0$ fix $r_j\in \C$ so that the $j$-th Laplace eigenvalue is $\lambda_j = \frac 14+r_j^2$. Let $h$ be an even function whose Fourier transform $\^h(u) = \frac 1{2\pi}\intinf h(r)e^{-iru}dr   $ is smooth and compactly supported, so that   $h(r) =  \intinf \^h(u)e^{iru}du$ is rapidly decaying and extends to an entire function.  Then
\[
\sum_{j=0}^\infty h(r_j) = (g-1) \intinf h(r)r \tanh(\pi r)dr + \sum_{\substack{\gamma \;{\rm primitive}\\{\rm oriented}}} \sum_{k=1}^\infty\frac{\ell_\gamma  \^h(k\ell_\gamma)}{2\sinh(k\ell_\gamma/2)} 
\]
where the sum is over all primitive {\em oriented} closed geodesics, equivalently over all nontrivial primitive conjugacy classes in the fundamental group of the surface. 

Now take $f$ even such that the Fourier transform $\^f\in C_c^\infty(\R)$ is smooth of compact support (and even). 
To fix ideas, lets assume that $\supp \^f = [-1,1]$.
 We take 
\[
h(r)=f\left(L\left(r-\tau\right)\right)+ f\left(L\left(r+\tau\right)\right)
\] 
which is even, with Fourier transform
\[
\^h(u) = \frac{2\cos(\tau u)}{L} \^f\left(\frac uL \right) 
\]
which is  smooth and compactly supported. We set
\[
N_{f,L,\tau}(X):= \sum_{j\geq 0} h(r_j) = \sum_{j\geq 0} f\left(L\left(r_j-\tau \right)\right) +  f\left(L\left(r_j + \tau \right)\right) .
\]
Note that if $\tau>0$ and $L\cdot \tau\gg 1$,   then the contribution of  second term summed over the eigenvalues $\lambda_j\geq 1/4$ can be shown to be negligible from Weyl's law. 
We have chosen to retain this symmetric form in part because it is convenient to directly use the Selberg trace formula here.  

Selberg's trace formula  allows us  to   decompose 
 \[
 N_{f,L,\tau}=\bar N+N^{osc}
 \]
with a ``smooth'' main term 
\[
\bar N=\bar N_{f,L,\tau}:=(g-1) \intinf \left\{ f\left(L\left(r-\tau\right)\right)+ f\left(L\left(r+\tau\right)\right) \right\}  r \tanh(\pi r)dr
\]
 which if $\intinf f(x)dx\neq 0$ is  asymptotic as $\tau\to \infty$, $L>1$  to 
\[
\bar N \sim (g-1) \intinf f(x)dx \frac{2\tau}{L} , 
\]
 and 
\begin{equation}\label{expand Nosc}
N^{osc}_{L}(\tau;X) = \frac 1{L} \sum_{\substack{\gamma \;{\rm primitive}\\{\rm oriented}}} \sum_{k\geq 1} 
\frac{\ell_\gamma}{ \sinh(k\ell_\gamma/2)} \^f\left(\frac{k\ell_\gamma}{  L}\right)  \cos(\tau k \ell_\gamma)
\end{equation}
where the sum is over all closed primitive  {\em oriented} geodesics $\gamma$, with $\ell_\gamma$ being the length. 
However, the summands do not depend on the orientation of the geodesics, so we can write
\begin{equation}\label{split N}
N^{osc} =2 \sum_{\gamma} H_L(\ell_\gamma)
\end{equation}
the sum over all primitive  {\em non-oriented} closed geodesics, where  
\[
H_L(x) =\frac{x}{L}\sum_{k=1}^\infty F(kx), \quad F(x) =   \frac{\^ f\left( \frac { x}{L}\right) \cos( x \tau)}{ \sinh( x /2)} ,
\]
the sum over $k\geq 1$ represents repetitions of a single primitive geodesic. 

 Note that using the Prime Geodesic Theorem \cite{Buser} allows us to bound $N^{osc}_{L}(\tau;X) \ll e^{L/2}/L$ so that when $L=o(\log \tau)$, we have $N_{f,L,\tau} \sim  (g-1) \intinf f(x)dx \frac{2\tau}{L}$. 
 
We split the sum \eqref{split N} over closed geodesics $\gamma$ taking into account the different types
of these geodesics, as 
\[
N^{osc} =  N_{\rm{sns}} + N_{\rm{SSep}} +N'
\]
where in $N_{sns}$   the sum runs over simple,  non-separating primitive closed geodesics, $N_{\rm SSep}$ is the sum over simple primitive geodesics which separate the surface into two connected components, 
and $N'$ is the sum over non-simple primitive geodesics;  all geodesics are {\em not} oriented. 
  

For the second moment, we write 
\[
(N^{osc})^2  = N_{{\rm sns},2} + N_{{\rm SSep},2}+N'' 
\]
 where $N_{{\rm sns},2}$ is the sum over pairs $(\gamma,\gamma')$ of identical or disjoint simple geodesics such that 
 $S\backslash \gamma\cup \gamma'$ is connected, $ N_{{\rm SSep},2}$ is the sum over pairs of identical or disjoint simple geodesics such that 
 $S\backslash \gamma\cup \gamma'$ is disconnected, and $N''$ is the sum over the remaining pairs of orbits, to be dealt with in \S~\ref{sec:nonsimple}.

\section{The expectation of $N^{osc}$}\label{sec:expectation}

Our goal in this section is to compute the expected value $\EWp(N^{osc})$: 
\begin{prop}\label{prop:expectation}
Fix $\tau>0$ and $L>1$. Then  
\[
\lim_{g\to \infty} \EWp\left(N^{osc}\right)  = \If  .
\]
\end{prop}

\subsection{Bounds for $H_L(x)$} 
We first study the function  
\[
H_L(x) =\frac{x}{L}\sum_{k=1}^\infty F(kx), \quad F(x) =   \frac{\^ f\left( \frac { x}{L}\right) \cos( x \tau)}{ \sinh( x /2)} .
\]

 \begin{lem}\label{lem:bounds for HL}
 
$H_L(x)$ vanishes for $x>L$, and is smooth in $(0,L]$. For $L>2$,   uniformly in $\tau$, 

i) For $0<x<1/2$, we have 
\begin{equation}\label{HL near 0}
|H_L(x) | \leq \frac 2L\left(   \log \frac 1x +O\left(1 \right)\right) .
\end{equation}

ii) For $x\geq 1/2$, we have   
$$|H_L(x) | \ll \frac xL   e^{-x/2} \cdot \mathbf 1_{[0,L]}(x).$$
\end{lem}

\begin{proof}
Observe that $H_L(x)$ vanishes for $x>L$ since $\^f$ is supported in $[-1,1]$, and for $0<x\leq L$, the sum  is finite, over $1\leq k\leq L/x$ and therefore $H_L(x)$ is smooth in $(0,L]$. 
We use a crude bound (recall $\^f$ is supported in $[-1,1]$). 
\[
|H_L(x)|\ll G_L(x):=\frac xL \sum_{k\leq  L/x} \frac 1{\sinh(kx/2)} .
\]

Assume first that  $0<x<1/2$. We use $\sinh(t)\geq t$ for $t=kx/2 \in(0,1/2)$, and 
\[ 
\sinh(t) = e^{t}\frac {1-e^{-2t}}{2} \geq e^t \frac {1-e^{-1/2}}{2} \geq  0.19  \cdot e^{t} 
\]
 for $t=kx/2 \geq 1/2$,  
to obtain
\[
G_L(x)\leq \frac xL \sum_{1\leq k<1/x} \frac 1{kx/2} +  \frac xL \sum_{1/x < k\leq L/x} \frac {6}{e^{kx/2}}  
\]
We have
\[
 \frac xL \sum_{1\leq k<1/x} \frac 1{kx/2}  = \frac 2L  \sum_{1\leq k<1/x} \frac 1k \leq \frac 2L\left( \log \frac 1x +1 \right), 
\]
on comparing the harmonic sum to an integral: 
$$\sum_{k=1}^N \frac 1k <\int_1^{N+1} \frac {dt}{t} +1.$$ 
For the second sum, we have since $x\in (0,1/2)$, 
\[
\frac xL \sum_{1/x < k\leq L/x} \frac {6}{e^{kx/2}}   \ll \frac xL  \frac 1{1-e^{-x/2}}  \ll \frac 1L
\]
so that we obtain 
\[
G_L(x) \leq  \frac 2L \left( \log \frac 1x +O(1) \right)   .
\]


For $x\geq 1/2$, use when $t=kx/2\geq 1/4$ that $\sinh(t)\geq  0.19  \cdot e^{t} $ as above, 
so that
\[
G_L(x)\ll \frac xL \sum_{k\geq 1}e^{-kx/2} \ll \frac xL  e^{-x/2}
\]
which is (ii). 
\end{proof}


\subsection{Properties of $\If$}
We define 
\[
\If:=\int_0^\infty H_L(x) \left(\frac{\sinh(x/2)}{x/2} \right)^2 x dx
\]
Since $H_L(x)$ is integrable near zero by \eqref{HL near 0} and vanishes for $x>L$, the integral is absolutely convergent and we can change summation and integration to write
\[
\If=\frac{4}{L}\int_0^\infty  \sum_{k=1}^\infty \^f\left(\frac {kx}L \right)\frac{\sinh^2(x/2)}{\sinh(kx/2)} \cos(k\tau x) dx.
\]

The following lemma shows that $\If=o(1)$  if $L\to \infty$, $L=o(\log \tau)$:
\begin{lem}\label{lem: int H_L}
For  $\tau>0$ and $L\gg 1$,  
\[
\If \ll   \frac{e^{L/2}}{L \tau }+\frac 1L \left( 1+\frac 1\tau \right)  .
\]
\end{lem}
\begin{proof}
We want to compute
\begin{multline*}
\int_0^\infty H_L(x)\left( \frac{\sinh(x/2)}{x/2} \right)^2 xdx  
\\
= \sum_{k\geq 1} \frac 1{L}\int_0^\infty x \frac{\^f\left(\frac{kx}{L} \right) \cos(kx\tau)}{ \sinh(kx/2)} 
\left( \frac{\sinh(x/2)}{x/2} \right)^2 xdx  
\\
=4\sum_{k\geq 1} \int_0^\infty  \^f(ky) \frac{ (\sinh Ly/2)^2}{\sinh (kLy/2)} \cos( k\tau Ly) dy 
\end{multline*}
 (after changing variable).
 
For $k=1$, 
we obtain 
\begin{multline*}
 4 \int_0^\infty \^f(y) \sinh(Ly/2)\cos(\tau Ly) dy =
\\
-\frac 4{\tau L} \int_0^\infty \left\{ (\^f)'(y) \sinh(Ly/2) +\frac L2 \^f(y) \cosh(Ly/2) \right\} \sin(\tau Ly) dy
\end{multline*}
after integration by parts. Taking absolute values using $|\sin|\leq 1$ and $\supp \^f \subset [-1,1]$ gives
\[
\ll \frac{e^{L/2}}{L \tau  } .
\]


For $k=2$, when $\tau>0$, we integrate by parts, noting that $ \frac{ \sinh(x/2)^2}{\sinh x}$ is bounded and vanishes at $x=0$, with derivative $1/(4 \cosh^2 (x/2)) \ll e^{-x}$, to obtain 
 \begin{multline*}
\frac 1L \int_0^\infty \^f\left(\frac{2x}{L}\right) \frac{ \sinh^2(x/2)}{\sinh x} \cos(2\tau x) dx 
\\
=
-\frac 1{2\tau L} \int_0^\infty \left\{ \^f\left(\frac{2x}{L}\right) \left(  \frac{ \sinh^2(x/2) }{\sinh x}\right)' 
+\frac 2L  (\^f)'\left(\frac{2x}{L}\right) \frac{ \sinh^2(x/2) }{\sinh x} \right\} \sin(2\tau x) dx .
\end{multline*}
 Taking absolute values gives
 \[
 \ll \frac 1{\tau L} \int_0^\infty \left|  \^f\left(\frac{2x}{L}\right)  \right|e^{- x}  dx  + \frac 1{\tau L^2 }\int_0^\infty \left|    
 (\^f)'\left(\frac{2x}{L}\right) \right|dx \ll \frac 1{\tau L} .
 \]

For $k> 2$  we use 
\[
\sinh(kz) = \sinh z \cosh((k-1)z) + \cosh z \sinh((k-1)z) >\frac 12 \sinh z e^{(k-1)z}
\]
so that 
\[
\frac{(\sinh z)^2}{\sinh k z} < \frac 12 \sinh z e^{-(k-1)z}
\]
and
\begin{multline*}
\sum_{k> 2}  \left|  \int_0^\infty  \^f(ky) \frac{ (\sinh  \frac{Ly}{2})^2}{\sinh kLy/2} \cos( k\tau Ly) dy \right| 
\\
\ll \int_0^1 \sinh  \frac{Ly}{2} \sum_{k\geq 3} e^{-(k-1) Ly/2} dy
\ll \frac 1L  \int_0^\infty  \sinh(z) \frac{e^{-2z}}{1-e^{-z}} dz   
\ll \frac 1L  .
\end{multline*}
\end{proof}

 
 \subsection{Proof of Proposition~\ref{prop:expectation}}

\begin{proof}
 It will suffice to show  
 \[
 \lim_{g\to \infty}\EWp(N_{sns})  =\If  
 \]
which we do below, and 
 \begin{equation} \label{orig first moment of N_{Ssep}}
 \EWp\left(  N_{\rm SSep}  \right)\ll_{\tau,L} \frac 1g, 
 \end{equation}
which we will  do in our treatment of the variance, in the course of the  proof of  Proposition~\ref{prop:separating pairs}, 
see  \eqref{first moment of N_{Ssep}}, and 
\begin{equation}\label{orig expect N'}
 \EWp\left( N' \right) \ll _{\tau,L} \frac 1g  ,
\end{equation}
which we will do in \S~\ref{sec:boundN' and nonsimplepairs},  
see \eqref{expect N'}.   

 Recall \eqref{split N}
\begin{equation*}
 N_{sns} =2 \sum_{\substack{\gamma\\ \rm simple \\ \rm non-separating}} H_L(\ell_\gamma) .
\end{equation*}

 By Mirzakhani's integration formula \eqref{MIF one var}, for $g>2$, 
\[
\EWp(N_{sns}) =2\cdot \frac 12    \cdot    \int_{0}^{ \infty } H_L(x)  \frac{ V_{g-1,2}( x,x)} {V_g}
 x dx  .
\]
This is because the sum over all simple non-separating geodesics amounts to taking the sum over a single orbit of the mapping class group, thus defining the geometric function associated with these orbits.
By \eqref{ratio V_{g,2}(x,x)/V_g} and \eqref{ratio V_{g-1,2}/V_g}
\[
\frac{V_{g-1,2}(x,x)}{V_{g}} =\frac{V_{g-1,2}(x,x)}{V_{g-1,2}} \cdot \frac{V_{g-1,2}}{V_g}
= \left( \frac {\sinh(x/2)}{x/2} \right)^2 \cdot \left( 1 +O\left( \frac {1+x^3}g \right) \right) .
\]
Therefore
\[
\lim_{g\to \infty} \EWp(N_{sns}) = \int_0^{\infty} H_L(x) \left( \frac {\sinh(x/2)}{x/2} \right)^2   x d x  =\If 
\]
 proving Proposition~\ref{prop:expectation}. 
\end{proof}

\section{The variance}\label{sec:variance}

Notice that the main term $\bar N$ is independent of the random geometry,
and will therefore  disappear from the variance. So the variance of $N_{f,L,\tau}$ coincides with  the variance of $N^{osc}$. 

We now want to compute the second moment of $N^{osc}$. 
We write 
\[
(N^{osc})^2  = N_{{\rm sns},2} + N_{{\rm SSep},2}+N'' 
\]
 where $N_{{\rm sns},2}$ is the sum over pairs $(\gamma,\gamma')$ of identical ($\gamma=\gamma'$) or disjoint (i.e. $\gamma\cap \gamma'=\emptyset$)  simple geodesics such that 
 $S\backslash \gamma\cup \gamma'$ is connected, $ N_{{\rm SSep},2}$ is the sum over pairs of identical or disjoint simple geodesics such that 
 $S\backslash \gamma\cup \gamma'$ is disconnected, and $N''$ is the sum over remaining pairs.

What we find is that  only the sum $N_{{\rm sns},2}$ over {\em simple, non-separating} geodesics contribute to the main term.  
We will show (Proposition~\ref{prop:sns})  that for fixed $\tau>0$, and $L\gg 1$,   
\[
\lim_{g\to \infty}\EWp \left(  N_{\rm{sns},2}   \right) = \Sigma^2_{\rm GOE}(f)  +\If^2  
+O \left(\frac 1{(\tau L)^2}+ \frac{\log L}{L^2} \right) 
\]
and that (Proposition~\ref{prop:separating pairs}) 
\[
 \EWp \left(    N_{\rm{SSep},2}  \right)   \ll_{L,\tau} \frac 1g  
\] 
and (Proposition~\ref{prop:expected sum nonsimple pairs})
\[
 \EWp \left(  N'' \right)    \ll_{L,\tau}  \frac 1{\sqrt{g}} .
\]
This will give for fixed $\tau>0$, 
\[
  \lim_{g\to \infty}\EWp \left( \left(N^{osc} \right)^2 \right) 
=   \Sigma^2_{\rm GOE}(f)+\If^2 +O_\tau \left( \frac{\log L}{L^2} \right) .
\]
Therefore
\[
\begin{split}
\Var(N_{f,L,\tau}) &= \Var(N^{osc}) = \EWp\left( \left(N^{osc} \right)^2 \right) - \left( \EWp \left(N^{osc}\right) \right)^2
\\
&\underset{g\to \infty}{\longrightarrow} 
\Sigma^2_{\rm GOE}(f)  + \If^2 + O \left( \frac{\log L}{L^2} \right) -  \If^2  
\\
&=\Sigma^2_{\rm GOE}(f)  +  O \left( \frac{\log L}{L^2} \right) 
\end{split}
\]
so that for fixed $\tau>0$, 
\[
\lim_{L\to \infty} \left( \lim_{g\to \infty} \EWp \left( |N_{f,L,\tau}-\EWp(N_{f,L,\tau}) |^2 \right)  \right) =  \Sigma^2_{\rm GOE}(f) 
\]
which proves Theorem~\ref{main thm}.

\subsection{Simple non-separating geodesics} 
The term $N_{\rm{sns},2} $ is given by
\[
  N_{\rm{sns},2}    =   4  \sum_{\gamma \;\; \rm{sns}}H_L\left(\ell_\gamma\left(X \right) \right)^2 + 
 4 \sum_{\substack{(\gamma,\gamma') \;\; \rm{sns} \\ \gamma\cap \gamma'=\emptyset}}H_L\left(\ell_\gamma\left(X \right) \right) 
H_L\left(\ell_{\gamma'}\left(X \right) \right)
\]
where the first sum (the diagonal pairs) is over simple, non-separating geodesics, and the second sum (off-diagonal pairs) is over pairs of disjoint simple geodesics $(\gamma,\gamma')$ so that $S_g\backslash \gamma\cup \gamma'$ is connected. 

\begin{prop}\label{prop:sns}
For $\tau>0$, $L\gg 1$, 
\[
\lim_{g\to \infty}\EWp \left( N_{\rm{sns},2}\right) =  \Sigma^2_{\rm GOE}(f)   +\If^2 +O\left(\frac 1{(\tau L)^2}+ \frac{\log L}{L^2} \right) .
\]
\end{prop}

We use Mirzakhani's integration formula  to evaluate the expected values over $\Mg$ of each of the two terms, that is the diagonal and off-diagonal sums. 
Proposition~\ref{prop:sns} will follow from Lemma~\ref{prop:diagonal sns} and Lemma~\ref{lem:sns nondiag}.

\begin{lem}\label{prop:diagonal sns}
For $\tau>0$,  
\[
 \lim_{g\to \infty}  \EWp \left(  \sum_{\gamma \;\; \rm{sns}}H_L\left(\ell_\gamma \right)^2 \right)  =  
  \frac 14  \Sigma^2_{\rm GOE}(f) + O\left(\frac 1{(\tau L)^2}+ \frac{\log L}{L^2} \right) 
\]
where the sum is over simple, non-separating  {\em non-oriented} geodesics. 
\end{lem}

\begin{proof}
For the diagonal term, as in \S~\ref{sec:expectation}, with $H_L$ replaced by $H_L^2$, use \eqref{MIF one var} for $g>2$,  \eqref{ratio V_{g,2}(x,x)/V_g} and \eqref{ratio V_{g-1,2}/V_g} to find
(recall the sum is over non-oriented geodesics)  
\[
\begin{split}
 \lim_{g\to \infty}  \EWp \left(  \sum_{\gamma \;\; \rm{sns}}H_L\left(\ell_\gamma \right)^2 \right)  
& =  \frac 12    \int_0^\infty H_L(\ell)^2 \left( \frac{\sinh (\ell/2)}{\ell/2} \right)^2 \ell d\ell
\\
& = \frac 12     \sum_{k_1,k_2\geq 1}  I_L(k_1,k_2)
\end{split}
\]
where
\begin{multline*}
 I_L(k_1,k_2):=\frac 1{L^2} \int_0^\infty \ell^2 F(k_1 \ell) F(k_2 \ell)  \left( \frac{\sinh (\ell/2)}{\ell/2} \right)^2 \ell d\ell 
 \\
  = 
 \frac 4{L^2}  \int_0^\infty \ell  \^f\left( \frac{k_1 \ell}{L} \right)  \^f\left( \frac{k_2 \ell}{L} \right)  \frac{\sinh^2 (\ell/2)}{\sinh(k_1 \ell/2) \sinh(k_2 \ell/2)}\cos(\tau k_1\ell)\cos(\tau k_2\ell) d\ell  .
\\
\end{multline*}

For $k_1=k_2=1$ we obtain 
\begin{equation}
\begin{split}
I_L(1,1) &= \frac 4{L^2}  \int_0^\infty \ell  \^f\left( \frac{  \ell}{L} \right)^2    \frac{1+\cos(2\tau \ell)}{2}d\ell
\\
 &=
    2 \int_0^\infty x \^f(x)^2 dx + 2 \int_0^\infty x \^f(x)^2 \cos(2\tau L x)dx 
   \\
   & =
      \intinf |x| \^f(x)^2 dx +O(1/(\tau L)^2)
 \end{split}
 \end{equation}
on using integration by parts twice to bound  the second term   (recall $\tau>0$). 

Next we show that the sum over $k_1+k_2\geq 3$ is $O(\log L/L^2 )$, uniformly in $\tau$, as $L\to \infty$.  
 We use for $  k_2\geq  k_1\geq 1 $,
\begin{equation}\label{upper bound on I(k)}
| I_L(k_1,k_2) | \ll \int_0^{1/k_2}  x \frac{\sinh(xL/2)^2}{\sinh(k_1xL/2)\sinh(k_2xL/2)}  dx .
\end{equation}

For $k\geq 1$, we have for $y>0$
\begin{equation}\label{bound for ratio of sinh}
  \frac{\sinh(y)}{\sinh(ky)}   <  \frac{ 2}{e^{ (k-1)y}} 
\end{equation}
since
\begin{multline*}
\sinh(ky) = \sinh(y)\cosh((k-1)y) + \cosh(y)\sinh((k-1))y) 
\\
>  \sinh(y)\cosh((k-1)y)  >\frac 12 \sinh (y) e^{(k-1)y} .
\end{multline*}
Therefore
\begin{equation}\label{bound for small k}
\left| I_L(k_1,k_2) \right| \ll \int_0^\infty  x e^{-(k_1+k_2-2)Lx} dx  =\frac 1{(k_1+k_2-2)^2L^2} .
\end{equation}

We will  use \eqref{bound for small k} for $k_1+k_2< L^2$. 
For $k_1+k_2\geq  L^2$, we instead use in \eqref{upper bound on I(k)}
\[
\frac{\sinh(y)}{\sinh(ky)}\leq \frac 1k
\]
to obtain
\begin{equation}\label{bound for large k}
| I_L(k_1,k_2) | \ll  \int_0^{1/k_2}  x \frac 1{k_1k_2} dx \ll \frac 1{k_1k_2^3} . 
\end{equation}

Summing \eqref{bound for small k} over $k\geq 2$ gives a bound for the diagonal terms
\[
\sum_{k\geq 2} I_L(k,k) \ll \frac 1{L^2} .
\]

Next we bound the sum over $1\leq k_1<k_2$: We divide the sum into two pieces, one over $k_1+k_2<L^2$ and the second over $k_1+k_2\geq L^2$. 
For the sum over $k_1+k_2<L^2$, we use \eqref{bound for small k}:
\begin{equation*}
\begin{split}
\sum_{3\leq k_1+k_2<L^2} | I_L(k_1,k_2) | &\ll  \sum_{3\leq k_1+k_2<L^2}\frac 1{(k_1+k_2-2)^2L^2}
\\
& \ll \frac 1{L^2} \sum_{1\leq m\leq L^2} \frac 1{m^2} \#\{(k_1,k_2):k_1+k_2=m\} 
\\
&\ll \frac 1{L^2} \sum_{1\leq m\leq L^2}  \frac 1{m^2}  \cdot m \ll \frac{\log L}{L^2} .
\end{split}
\end{equation*}

For the sum over $k_1+k_2> L^2$, we use \eqref{bound for large k} to find
\begin{equation*}
\begin{split}
\sum_{k_1+k_2 > L^2} | I_L(k_1,k_2) | &\ll 
\sum_{  k_1+k_2>L^2}\frac 1{k_1 k_2^3} 
\\
&\ll \sum_{1\leq k_1 \leq L^2/2} \frac 1{k_1} \sum_{k_2>L^2-k_1} \frac 1{k_2^3}
+ \sum_{  k_1 > L^2/2} \frac 1{k_1} \sum_{k_2>k_1} \frac 1{k_2^3}
\\
&\ll \sum_{1\leq k_1 \leq L^2/2} \frac 1{k_1} \frac 1{(L^2-k_1)^2}  +  \sum_{  k_1 > L^2/2} \frac 1{k_1} \cdot  \frac 1{k_1^2}
\\
&\ll \frac{\log L}{L^4} + \frac 1{L^4}\ll \frac{\log L}{L^4} .
\end{split}
\end{equation*}
Altogether we find that as $g\to \infty$,
\begin{multline}\label{ans for diag}
  \EWp \left(  \sum_{\gamma \;\; \rm{sns}}H_L\left(\ell_\gamma \right)^2 \right)  \sim  \frac 12   \int_0^\infty H_L(\ell)^2 \left( \frac{\sinh (\ell/2)}{\ell/2} \right)^2 \ell d\ell
 \\
 = \frac 12    \intinf |x| \^f(x)^2 dx + O\left( \frac{\log L}{L^2} \right) ,
\end{multline}
which gives the result on recalling that $\Sigma^2_{\rm GOE}(f) = 2 \intinf |x| \^f(x)^2 dx $. 
\end{proof}


We now bound the contribution of the off-diagonal non-separating pairs $(\gamma,\gamma')$ with $\gamma, \gamma'$ being disjoint  simple   geodesics such that  $S_g\backslash (\gamma\cup \gamma')$ is connected:
\begin{lem}\label{lem:sns nondiag}
\[
\lim_{g\to \infty}\EWp \left( 4 \sum_{\substack{(\gamma,\gamma') \;\; \rm{sns} \\ \gamma\cap \gamma'=\emptyset}}H_L\left(\ell_\gamma\left(X \right) \right) 
H_L\left(\ell_{\gamma'}\left(X \right) \right) \right) =\If^2 .
\]
\end{lem}
\begin{proof}
There is a single topological type (i.e. orbit of the mapping class group) of such pairs, giving $S_g\backslash (\gamma\cup \gamma') = S_{g-2,4}$ a surface of genus $g-2$ with two pairs of equal length boundary geodesics, see  e.g.  \cite[\S 1.3.1]{FM} and 
Figure~\ref{fig:two bdry}.   
Mirzakhani's integration formula \eqref{gen M integ formula} gives, 
\begin{multline*}
\EWp \left( \sum_{\substack{(\gamma,\gamma') \;\; \rm{sns} \\ \gamma\cap \gamma'=\emptyset}}H_L\left(\ell_\gamma  \right) 
H_L\left(\ell_{\gamma'} \right) \right)   
\\
= \frac 1{2^2} \int_0^\infty \int_0^\infty H_L(x)H_L(y) \frac{V_{g-2,4}(x,x,y,y)}{V_g}\cdot  xdx \cdot ydy .
\end{multline*}

Using \eqref{ratio V_{g,2}(x,x)/V_g} and \eqref{ratio V_{g-1,2}/V_g} we have  
\[
\begin{split}
 \frac{V_{g-2,4}(x,x,y,y)}{V_g}  =  \frac{V_{g-2,4}(x,x,y,y)}{V_{g-2,4}} \frac{V_{g-2,4}}{V_{g-1,2}} \frac{V_{g-1,2}}{V_g}
 \\
= \left( \frac{\sinh(x/2)}{x/2}\frac{\sinh(y/2)}{y/2}\right)^2 \left( 1+ \frac{ (x+y)x^2y^2 }{g} \right) .
\end{split}
 \]
 
Since $H_L(x)$ is compactly supported, and satisfies \eqref{HL near 0} near $x=0$, we may pass to the limit $g\to \infty$ and obtain
 \begin{multline*}
\lim_{g\to \infty}\EWp \left( \sum_{\substack{(\gamma,\gamma') \;\; \rm{sns} \\ \gamma\cap \gamma'=\emptyset}}H_L\left(\ell_\gamma  \right) 
H_L\left(\ell_{\gamma'} \right) \right) 
\\
=\lim_{g\to \infty}  \frac 1{2^2} \int_0^\infty \int_0^\infty H_L(x)H_L(y) \frac{V_{g-2,4}(x,x,y,y)}{V_g}\cdot  xdx \cdot ydy
\\
= \frac 14 \int_0^\infty\int_0^\infty H_L(x)H_L(y) \left( \frac{\sinh(x/2)}{x/2}\frac{\sinh(y/2)}{y/2}\right)^2 xdx \cdot ydy
\\
=  \frac 14  \left(  \int_0^\infty H_L(x)\left( \frac{\sinh(x/2)}{x/2} \right)^2 xdx \right)^2  =\frac 14 \If^2 
\end{multline*}
as claimed. 
\end{proof}

\begin{figure}[ht]
\begin{center}
\includegraphics[height=50mm, width=120mm]{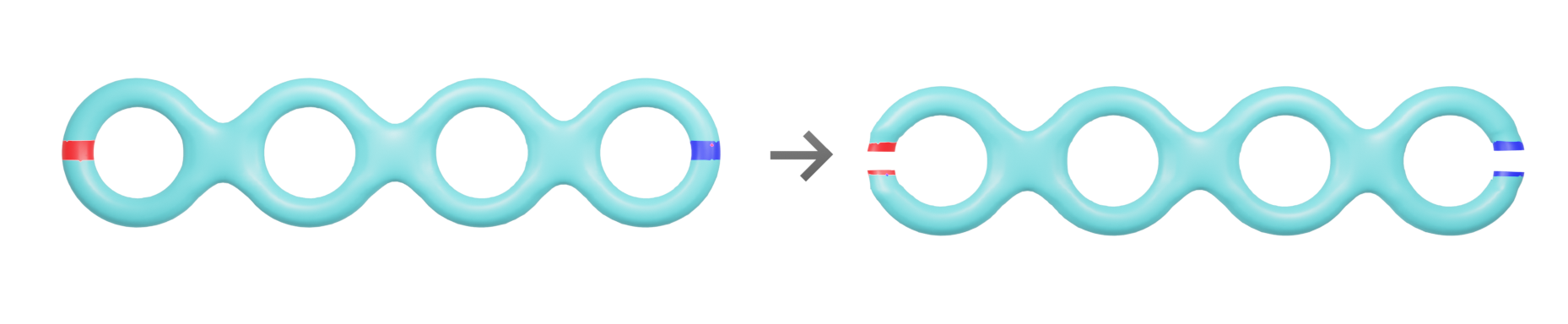}
\caption{ A genus 4 surface $S_{4,0}$ cut by a non-separating bi-curve to form a surface $S_{2,4}$ of genus two with two pairs of boundary components.}
\label{fig:two bdry}
\end{center}
\end{figure}

\subsection{Separating pairs}
The term $N_{{\rm SSep},2}$ is the sum
\[
\begin{split}
  N_{{\rm SSep},2}  &=4\sum_{\gamma \; {\rm separating}}  H_L(\gamma)^2  +
   4\sum_{\substack{\gamma \cup \gamma' \; {\rm separating} \\ \gamma\cap \gamma'=\emptyset}} H_L(\ell_\gamma(X)) H_L(\ell_{\gamma'}(X))
   \\ &=:D +O .
\end{split}
\]
where the first sum is over simple geodesics $\gamma$ so that $S_g\backslash \gamma$ is disconnected, and the second sum over pairs $(\gamma,\gamma')$ of disjoint simple geodesics so that $S_g\backslash \gamma\cup \gamma'$ is disconnected. 
We now show that the contribution of separating pairs is negligible: 
\begin{prop}\label{prop:separating pairs} 
For $L>1$,
\[
  \EWp( N_{{\rm SSep},2}  )   \ll_{L,\tau} \frac 1g  .
\]
\end{prop}
\begin{proof}

For the diagonal case, we collect together the separating geodesics of type $i$, that is $S_g\backslash \gamma =S_{i,1} \cup S_{g-i,1}$ 
is a union of two surfaces of genera $i\geq 1$ and $g-i\geq 1$, each having one boundary component (both of the same length),  for $1\leq i\leq \lfloor g/2 \rfloor$. Mirzakhani's integration formula gives  
\[
\EWp(D ) =4\sum_{i=1}^{\lfloor g/2 \rfloor} c_i \frac 1{V_g}\int_{0}^{ \infty } H_L(\ell)^2   V_{i,1}(\ell) V_{g-i,1}(\ell) \ell d\ell   
\]
for $0\leq c_i\leq 1$. 
 
We have
\[
\frac{V_{i,1}(\ell) V_{g-i,1}(\ell)}{V_g} = \frac{V_{i,1}(\ell)}{V_{i,1}}\frac{ V_{g-i,1}(\ell)}{V_{g-i,1}} \frac{V_{i,1}V_{g-i,1}} {V_g} .
\]
By   \eqref{ratio V_{g,2}(x,x)/V_g}, for $0\leq \ell \leq L$,  and $L>1$,  
\[
 \frac{V_{j,1}(\ell)}{V_{j,1}}   =  \frac{\sinh(\ell/2)}{\ell/2} \left(1+O\left(\frac{\ell^2}{j} \right) \right) \ll \frac{\sinh(\ell/2)}{\ell/2}  L^2 
\]
and for $g\geq 4$ (which we may assume), by \eqref{Gen bound on sum of vols} 
\[
\sum_{i=1}^{g-1} \frac{V_{i,1}V_{g-i,1}} {V_g}   \ll \frac 1{g} .
\]
Therefore for $0\leq \ell \leq L$, $L>1$, 
\[
\sum_{i=1}^{\lfloor g/2 \rfloor} \frac{V_{i,1}(\ell) V_{g-i,1}(\ell)}{V_g} \ll \frac {L^4}{g}   \left( \frac{\sinh(\ell/2)}{\ell/2} \right)^2 .
\]
Hence 
\[
  \EWp(D) \ll_L \frac 1g \int_0^\infty H_L(\ell)^2   \left( \frac{\sinh(\ell/2)}{\ell/2} \right)^2 \ell d\ell .
\]
Comparing with \eqref{ans for diag} gives
\[
  \EWp(D)\ll_{L,\tau} \frac 1g .
\]

The same consideration works for the first moment of $N_{\rm SSep}$ and gives
\begin{equation}\label{first moment of N_{Ssep}}
\begin{split}
\EWp(N_{\rm SSep}) &=  \sum_{i=1}^{\lfloor g/2 \rfloor} c_i \frac 1{V_g}\int_{0}^{ \infty } H_L(\ell)    V_{i,1}(\ell) V_{g-i,1}(\ell) \ell d\ell   
\\
&\ll_L \frac 1g \int_0^\infty | H_L(\ell)|   \left( \frac{\sinh(\ell/2)}{\ell/2} \right)^2 \ell d\ell \ll_{L,\tau} \frac 1g 
\end{split}
\end{equation}
 proving \eqref{orig first moment of N_{Ssep}}.

 To treat  the off-diagonal pairs,  we break them up according to their topological type, that is the orbit under the mapping class group.
This is determined by the topology of the complement $S_g\backslash \gamma\cup \gamma'$, 
firstly by the number $q\geq 2$ of connected components, and then by the topological type of the possible components: 
$S_g\backslash \gamma\cup \gamma'=\cup_{i=1}^q S_{g_i,b_i}$ where $S_{g_i,b_i}$  of genus $g_i$ and having $b_i \geq 1$ boundary components. 

Necessarily the total number of boundary components is $\sum_{i=1}^q b_i=4$, so that $q\in \{2,3,4\}$.   We cannot have all $b_i=1$, since having $4$ pieces with 
one  boundary component each would impose their gluing along the boundaries to result in two disconnected
surfaces, while $S_g$ is  connected, so necessarily $q\in \{2,3\}$. 
For $q=2$ we have either $b_1=1$, $b_2=3$ or $b_1=b_2=2$. For $q=3$ we must have $b_1=b_2=1$, $b_3=2$. 
Moreover,  
by the additivity of the Euler characteristic, we have
\[
2-2g = \sum_{i=1}^q (2-2g_i-b_i)
\]
giving 
\[
\sum_{i=1}^q g_i  = g+q-3 .
\] 
For instance, for  $q=3$, when  $S_g \backslash \gamma\cup \gamma' = S_{i,1}(x)\cup S_{j,2}(x,y)\cup S_{k,1}(y)$ is a union of three surfaces with  matching boundary components as in Figure~\ref{fig:Three pieces}, of genera $i,j,k\geq 1$ adding up to $g$: $i+j+k=g$. 
\begin{figure}[ht]
\begin{center}
\includegraphics[height=50mm, width=120mm]{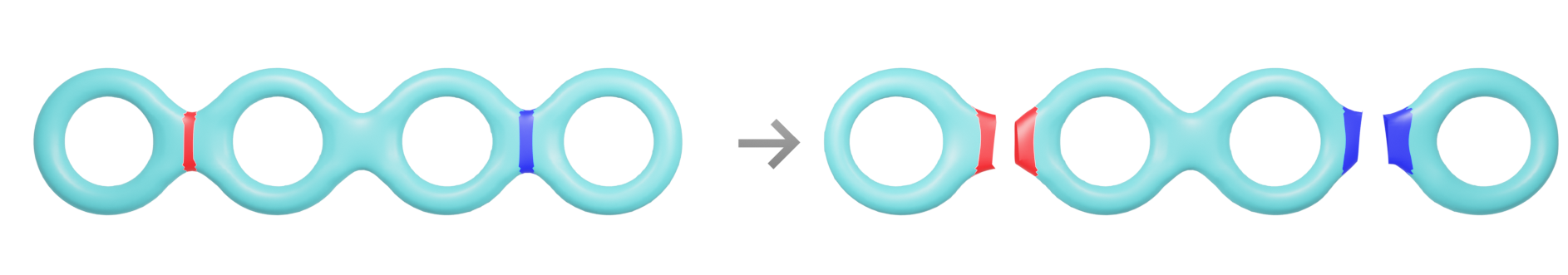}
\caption{ A genus 4 surface $S_{4,0}$ cut by a pair of nonisotopic  separating curves to form three surfaces with boundary $S_{1,1}$, $S_{2,2}$ and $S_{1,1}$.}
\label{fig:Three pieces}
\end{center}
\end{figure}

 To bound the contribution   of a given orbit $\mathcal O$ of pairs $(\gamma,\gamma')$, we use \eqref{gen M integ formula} to obtain, for $0< c(\mathcal O)\leq 1$, 
 \begin{multline}\label{int over orbit}
 \EWp \left(\sum_{(\gamma,\gamma')\in \mathcal O} H_L(\ell_\gamma)H_L(\ell_\gamma') \right)  
 \\ = 
 \frac {c(\mathcal O)}{V_g} \iint H_L(x) H_L(y) \prod_{i=1}^{q(\mathcal O)} V_{g_i,b_i}( \vec x) xdx ydy 
 \end{multline}
 where $\vec x$ is the vector of lengths of the boundary components $b_i$, two being equal to $x$ and the other two being equal to $y$, and $q(\mathcal O)=2$ or $3$. 
 For instance, when $q=3$ so that $b_1=1$, $b_2=2$ and $b_3=1$  as in Figure~\ref{fig:Three pieces}, then 
 \[
 S_g\backslash \gamma\cup \gamma' = S_{g_1,1}(x)\cup S_{g_2,2}(x,y) \cup S_{g_3,1}(y)
 \] 
 and the factor $\prod_{i=1}^q V_{g_i,b_i}( \vec x) $ is $V_{g_1,1}(x)V_{g_2,2}(x,y) V_{g_3,1}(y)$. 

  For the general upper bound, we use \eqref{ratio V_{g,2}(x,x)/V_g} in \eqref{int over orbit} to replace 
 \[
  \prod_{i=1}^{q(\mathcal O)} V_{g_i,b_i}( \vec x) = \prod_{i=1}^{q(\mathcal O)} V_{g_i,b_i}  \cdot \prod_{j=1}^4\frac{\sinh(x_j/2)}{x_j/2}\left(1+O\left(\frac{\sum_{j=1}^4 x_j \prod_{j=1}^4 x_j }{g_i} \right) \right) .
  \]
  Now note that  two of the $x_j$ equal $x$ and the other two equal $y$, and that the range of integration is for $0\leq x,y\leq L$ since $H_L(x)$ vanishes when $x>L$. Hence we may bound the integrand in 
  \eqref{int over orbit} by 
  \[
  \ll |H_L(x) H_L(y)|  \cdot \prod_{i=1}^{q(\mathcal O)} V_{g_i,b_i}  \cdot   \left( 1 + O\left(\frac{L^5}{g_i} \right) \right)^4.  
  \]
  Since $H_L(x)$ is integrable at $x=0$, and smooth in $(0,L]$  (Lemma~\ref{lem:bounds for HL}), we find
  \[
  \EWp \left(\sum_{(\gamma,\gamma')\in \mathcal O} H_L(\ell_\gamma)H_L(\ell_\gamma') \right)  \ll_L  \frac 1{V_g}\prod_{i=1}^{q(\mathcal O)} V_{g_i,b_i} .
  \]
  
  Summing over all orbits $\mathcal O$ we obtain a bound of 
       \[
\sum_{\mathcal O} \EWp \left(\sum_{(\gamma,\gamma') \in \mathcal O} H_L(\ell_\gamma)H_L(\ell_\gamma') \right)  \ll_L \frac 1{V_g} \sum_{\mathcal O}  \prod_{i=1}^{q(\mathcal O)} V_{g_i,b_i}  .
 \]
On using \eqref{Gen bound on sum of vols}, this is $\ll_L 1/g$, so we obtain
\[
\EWp(O)
 = \EWp\left(4 \sum_{\substack{ \gamma\cup \gamma' \;{\rm separating}\\ \gamma\cap \gamma'=\emptyset}} 
H_L(\ell_\gamma(X)) H_L(\ell_{\gamma'}(X)) \right) 
\ll_L \frac 1g .
\]
Therefore we showed
\[
\EWp(N_{SSep}^2)\ll_L \frac 1g .
\]
 %
%
 \end{proof}

\section{Bounding the non-simple case}\label{sec:nonsimple}

 \subsection{A collar lemma and its applications} 
We first recall that there is a uniform lower bound on the length of a non-simple geodesic, and a lower bound on the length of any closed geodesic intersecting a simple short geodesic:

\begin{lem}\label{lem:Basmajian}
i) Any non-simple geodesic has length at least $4\arcsinh(1)=3.52\dots$. 

ii) Let $\gamma, \gamma'$ be a pair of distinct intersecting closed geodesics, having lengths $\ell=\ell_\gamma$ and $\ell'=\ell_{\gamma'}$, with $\gamma$ simple. If $\ell \leq  1/2$ then 
\[
\ell' >2\log \coth(\frac \ell 4) 
\]
so that $\ell'  \gg \log \frac 1\ell $ as $\ell \searrow 0$. 
\end{lem}
\begin{proof}
 i) There is a uniform lower bound on the length of a non-simple geodesic, in fact the length is at least $\ell_{\min}=4\arcsinh(1)=3.52\dots$ (and this is sharp), see \cite[Ch. 4 \S 2]{Buser}. 

ii) We use \cite[Corollary 4.1.2]{Buser} which says that in this situation, where the intersections are guaranteed to be transversal as the geodesics are distinct, we have
\[
\sinh(\ell/2) \cdot \sinh (\ell'/2)>1 .
\]
Hence 
\[
\ell'>2\sinh^{-1}( \frac 1{\sinh(\ell/2) })  = 2\log \coth(\ell/4) =2 \log \left(\frac 4\ell\right) +O(\ell^2) .
\]
\end{proof}

We claim that for any pair of intersecting geodesics $(\gamma,\gamma')$, we have a uniform upper bound on the product $|H_L(\ell_{\gamma})| \cdot |H_L(\ell_{\gamma'})|$. 


\begin{prop}\label{prop:bound non-simple pair}
i) Let $\gamma$ be a non-simple geodesic. Then 
\[
|H_L(\ell_\gamma)|\ll \frac 1L  \mathbf 1_{[0,L]}(\ell_\gamma).
\] 
Hence if both $\gamma$ and $\gamma'$ are non-simple closed geodesics, then there is a uniform bound
\[
|H_L(\ell_{\gamma}) \cdot  H_L(\ell_{\gamma'}) | \ll   \frac 1{L^2}\mathbf 1_{[0,L]}(\ell_{\gamma})\cdot  \mathbf 1_{[0,L]}(\ell_{\gamma'}) .
\]

ii) Let  $(\gamma,\gamma')$ be pair of intersecting closed geodesics on a hyperbolic surface, at least one of them simple. Then 
\[
|H_L(\ell_{\gamma}) \cdot  H_L(\ell_{\gamma'}) | \ll_L \mathbf 1_{[0,L]}(\ell_{\gamma})\cdot  \mathbf 1_{[0,L]}(\ell_{\gamma'})  .
\]
\end{prop}

\begin{proof}
i) If $\gamma$ is non-simple, then by Lemma~\ref{lem:Basmajian}(i) we have $\ell_\gamma\gg 1$, 
and by Lemma~\ref{lem:bounds for HL}(ii) we deduce $|H_L(\ell_\gamma)|\ll_L \mathbf 1_{[0,L]}$. 


ii) Now assume that $\gamma$ is simple of length $\ell=\ell_{\gamma}$, and that $\gamma'\neq \gamma$ intersects $\gamma$, and denote its length by $\ell'=\ell_{\gamma'}$. If $\gamma'$ is also simple then assume, as we may, that $\ell\leq \ell'$. 

If $\ell>1/2$ then if $\gamma'$ is also simple then also $\ell'\geq \ell >1/2$, while if  $\gamma'$ is non-simple then by 
Lemma~\ref{lem:Basmajian}(i) we still have $\ell'\gg 1$, and so in both these cases we have individual bounds $|H_L(\ell)|\ll 1$ and $|H_L(\ell')|\ll 1$ by part (i), so the product is bounded.

If $\ell <1/2$  then by Lemma~\ref{lem:Basmajian}(ii) we have $\ell'>4$, so that by Lemma~\ref{lem:bounds for HL}(ii) we have 
$|H_L(\ell')|  \ll  \ell' e^{-\ell'/2}$. Furthermore,   by Lemma~\ref{lem:Basmajian}(ii)  we know 
  $\ell' >2\log \coth(\frac \ell 4)$, and since $\ell' e^{-\ell'/2}$ is decreasing for $\ell'>4$ we have 
\[
\begin{split}
|H_L(\ell')| & \ll  \ell' e^{-\ell'/2} < 2\log \coth(\frac \ell 4) \exp( - \log \coth(\frac \ell 4)) 
\\
&= \frac{2\log \coth(\frac \ell 4) }{ \coth(\frac \ell 4)} \sim \frac \ell 2 \log \frac 4 \ell
\end{split}
\]
as $\ell \to 0$, and since $\ell <1/2$ then by Lemma~\ref{lem:bounds for HL}(i), $|H_L(\ell)|\ll \log(1/\ell)$ and we find
\[
|H_L(\ell ) \cdot  H_L(\ell') | \ll \ell ( \log \frac 1 \ell )^2  \leq \left(\frac 2 e \right)^2 =O(1) 
\]
 (using $\ell<1$). 
\end{proof}

 \subsection{Bounding $\EWp(N') $}
 \label{sec:boundN' and nonsimplepairs}

We first prove the bound   \eqref{orig expect N'} on the expected value of $N'$: from Proposition~\ref{prop:bound non-simple pair}(i) we reduce to bounding the expected number of non-simple closed geodesics:  
\begin{equation}\label{expect N'}
\begin{split}
\EWp(N') & \leq \EWp\left(\sum_{\gamma \;{\rm non\; simple}} |H_L(\ell_\gamma)| \right)
\\
&\ll_L \EWp\left( \sum_{\gamma \;{\rm non\; simple}} 1\right)  \ll \frac 1g 
\end{split}
\end{equation}
by \cite[Proposition 4.5]{MP}, proving \eqref{orig expect N'}.

\subsection{Bounding $N''$}  
We consider the contribution $N''$ of all pairs of closed geodesics $(\gamma,\gamma')$ which form a non-simple pair, which means that 
at least one of the following (possibly overlapping) conditions hold: 
\begin{itemize}
\item The diagonal case: $\gamma=\gamma'$ is not simple.  
\item Both geodesics are non simple and distinct.
\item the geodesics are distinct $\gamma\neq \gamma'$ but intersect (possibly both are simple). 
\item the geodesics are disjoint $\gamma\cap \gamma'=\emptyset$, one is simple and the other is non-simple.

\end{itemize}
Then 
\[
|N''|\leq  \sum_{\gamma\; {\rm non \; simple}} |H_L(\ell_\gamma)|^2 
+ \sum_{ (\gamma,\gamma') \; {\rm non-simple\; pair} } 
 | H_L(\ell_{\gamma})H_L(\ell_{\gamma'})| .
\]
The bound on the contribution to $\EWp(N'')$ of diagonal pairs $\gamma=\gamma'$ non-simple is exactly as in the bound for $\EWp(N')$, 
see \eqref{expect N'}.  

We will show that the expected value of the contribution of non-diagonal non-simple pairs is bounded by 
\begin{prop}\label{prop:expected sum nonsimple pairs}
Fix $\tau>0$, $L>1$. Then 
\[
\EWp\left(  \sum_{\substack{\gamma\neq \gamma'\\ (\gamma,\gamma') \; {\rm non-simple\; pair}}} 
 | H_L(\ell_{\gamma})H_L(\ell_{\gamma'})|  \right) \ll_{L,\tau} \frac 1{\sqrt{g}}. 
\]
\end{prop}

 We bound the expected value of the sum over pairs of distinct  geodesics $\gamma\neq \gamma'$, either both non-simple or  intersect each other: Proposition~\ref{prop:bound non-simple pair}   allows us to bound the expected value of the sum over all such pairs  in terms of the expected value of the  number  $Y'_{g,2}$ of such pairs of geodesics of length at most $L$, which was bounded as $O_L(1/g)$ in \cite[Proposition 4.5]{MP}, and so we obtain
\begin{equation}\label{Sum of intersecting pairs}
\begin{split}
\EWp\left( \sum_{\substack{\gamma\neq \gamma'\\ \gamma\cap \gamma' \neq \emptyset \\ {\rm or\; both\; non-simple}  }}  
| H_L(\ell_{\gamma})H_L(\ell_{\gamma'})| \right) 
&\ll_L  \EWp\left( \sum_{\substack{\gamma\neq \gamma'\\ \gamma\cap \gamma' \neq \emptyset   \\ {\rm or\; both\; non-simple}  }} 1 \right) 
\\
&= \EWp\left( Y'_{g,2} \right)\ll \frac 1g .
\end{split}
\end{equation}

We are left to treat  the case that one of the geodesics is simple, the other non-simple, but they do not intersect.
We bound this sum by the sum over all pairs $(\gamma,\gamma')$ where $\gamma$ is simple and $\gamma'$ is non-simple, which splits into a product of individual sums: The sum $A$ over simple $\gamma$ and  the sum $B$ over non-simple $\gamma'$:
\[
\sum_{\gamma \; \rm{simple}} |H_L(\ell_\gamma)| \sum_{\gamma'\;{\rm non-simple}} |H_L(\ell_{\gamma'})|=:A \cdot B .
\]
Using the bound on $ |H_L(\ell_{\gamma'})|\ll_L \mathbf 1_{[0,L]}(\ell_{\gamma'})$ for non-simple $\gamma'$ 
  of Proposition~\ref{prop:bound non-simple pair} (i) gives
\[
B\ll_L Y'
\]
where $Y'$ is the number of non-simple geodesics of length at most $L$. Using Cauchy-Schwarz, we find
\[
\EWp(A\cdot B) \ll_L \EWp(A\cdot Y') \leq \sqrt{\EWp(A^2)} \cdot \sqrt{\EWp((Y')^2)} .
\]

The second moment of $Y'$ was bounded in  \cite[Proposition 4.5]{MP} by
\[
\EWp((Y')^2)=\EWp((Y'(Y'-1)) + \EWp(Y') \ll\ \frac 1g.
\]
We claim that the second moment   $\EWp(A^2)$ is uniformly bounded: To see this, expand 
\[
A^2 = \sum_{\substack{\gamma,\gamma' \;{\rm simple}\\ {\rm non-intersecting} }} |H_L(\ell_{\gamma})H_L(\ell_{\gamma'})| + 
  \sum_{\substack{\gamma,\gamma' \;{\rm simple}\\ {\rm  intersecting} }} |H_L(\ell_{\gamma})H_L(\ell_{\gamma'})|   .
\]
The expected value of the sum over pairs of simple, non-intersecting pairs was already shown to be \underline{bounded} (Propositions \ref{prop:sns}, \ref{prop:separating pairs}). 
The expected value of the sum over intersecting pairs was shown in \eqref{Sum of intersecting pairs} to be $O_L(1/g)$. Therefore we obtain
\[
\EWp(A^2) \ll_{L,\tau} 1
\]
and so 
\[
\EWp(A\cdot B) \ll_L \EWp(A^2)^{1/2} \cdot \EWp((Y')^2)^{1/2} \ll_{L,\tau}   \frac 1{\sqrt{g}}
\]
proving Proposition~\ref{prop:expected sum nonsimple pairs}. \qed

\end{document}